\documentclass{amsart}%
\usepackage{amsfonts}
\usepackage{amsmath}
\usepackage{amssymb}
\usepackage{graphicx}%
\providecommand{\U}[1]{\protect\rule{.1in}{.1in}}
\newtheorem{theorem}{Theorem}
\theoremstyle{plain}

\newtheorem{corollary}{Corollary}

\newtheorem{definition}{Definition}
\newtheorem{example}{Example}

\newtheorem{lemma}{Lemma}

\newtheorem{proposition}{Proposition}
\newtheorem{remark}{Remark}

\numberwithin{equation}{section}
\begin{document}
\title[ON WEAKLY $S$-PRIMARY SUBMODULES]{ON WEAKLY $S$-PRIMARY SUBMODULES}
\author{Ece YETKIN\ CELIKEL}
\address{Department of Electrical-Electronics Engineering, Faculty of Engineering,
Hasan Kalyoncu University, Gaziantep, Turkey.}
\email{ece.celikel@hku.edu.tr, yetkinece@gmail.com}
\author{Hani A. Khashan }
\address{Department of Mathematics, Faculty of Science, Al al-Bayt University,Al
Mafraq, Jordan.}
\email{hakhashan@aabu.edu.jo}
\thanks{This paper is in final form and no version of it will be submitted for
publication elsewhere.}

\begin{abstract}
Let $R$ be a commutative ring with a non-zero identity, $S$ be a
multiplicatively closed subset of $R$ and $M$ be a unital $R$-module. In this
paper, we define a submodule $N$ of $M$ with $(N:_{R}M)\cap S=\emptyset$ to be
weakly $S$-primary if there exists $s\in S$ such that whenever $a\in R$ and
$m\in M$ with $0\neq am\in N$, then either $sa\in\sqrt{(N:_{R}M)}$ or $sm\in
N$. We present various properties and characterizations of this concept
(especially in finitely generated faithful multiplication modules). Moreover,
the behavior of this structure under module homomorphisms, localizations,
quotient modules, cartesian product and idealizations is investigated.
Finally, we determine some conditions under which two kinds of submodules of
the amalgamation module along an ideal are weakly $S$-primary.

\end{abstract}
\maketitle

\section{Introduction}

Throughout this article, all rings are commutative with identity and all
modules are unital. Let $R$ be a ring and let $M$ be an $R$-module. A
non-empty subset $S$ of a ring $R$ is said to be a multiplicatively closed set
if $S$ is a subsemigroup of $R$ under multiplication. For a submodule $N$ of
$M$, we will denote by $(N:_{R}M)$ the residual of $N$ by $M$, that is, the
set of all $r\in R$ such that $rM\subseteq N$. $M$ is called a multiplication
module if every submodule $N$ of $M$ has the form $IM$ for some ideal $I$ of
$R$. Let $N$ and $K$ be submodules of a multiplication $R$-module $M$ with
$N=IM$ and $K=JM$ for some ideals $I$ and $J$ of $R$. The product of $N$ and
$K$ denoted by $NK$ is defined by $NK=IJM$. In particular, for $m_{1},m_{2}\in
M$, by $m_{1}m_{2},$ we mean the product of $Rm_{1}$ and $Rm_{2}$,
\cite{Ameri}. We call $M$ faithful if it has a zero annihilator in $R$, that
is $(0:_{R}M)=0$.

A proper submodule $N$ of $M$ is said to be prime (resp. primary) if whenever
$r\in R$ and $m\in M$ such that $rm\in N$, then $r\in(N:_{R}M)$ (resp.
$r\in\sqrt{(N:_{R}M)}$) or $m\in N$. For any submodule $N$ of an $R$-module
$M$ the radical, $M-rad(N)$, of $N$ is defined to be the intersection of all
prime submodules of $M$ containing $N$, \cite{El-bast}. It is shown in
\cite[Lemma 2.4]{2-abs} that if $N$ is a proper submodule of a multiplication
$R$-module $M$, then $M-rad(N)=\sqrt{(N:_{R}M)}M=\{m\in M$ $|$ $m^{k}\subseteq
N$ for some $k\geq0\}$. Moreover, we have $(M$-$rad(N):_{R}M)=\sqrt{(N:_{R}%
M)}$ for any finitely generated multiplication module $M.$The concepts of
prime and primary submodules have been generalized in several ways (see, for
example, \cite{WS-prime}, \cite{wp}, \cite{S-prime}, \cite{2-abs},
\cite{S-prime subm}, \cite{WS-primary} and \cite{HaniEce}).

In \cite{S-prime subm}, the authors introduced the concept of $S$-prime
submodules and investigate many properties of this class of submodules. More
generally, the concept of weakly $S$-prime ideals has been recently studied in
\cite{HaniEce}. Let $S$ be a multiplicatively closed subset of a ring $R$ and
$N$ be a submodule of an $R$-module $M$ such that $(N:_{R}M)\cap S=\phi$. Then
$N$ is called an $S$-prime (resp. weakly $S$-prime) submodule if there exists
$s\in S$ such that for $a\in R$ and $m\in M$, if $am\in N$ (resp. $0\neq am\in
N$), then $sa\in(N:_{R}M)$ or $sm\in N$. In 2021, Farshadifar, \cite{Farsh}
defined a submodule $N$ of $M$ with $(N:_{R}M)\cap S=\emptyset$ to be an
$S$-primary submodule if there exists a fixed $s\in S$ and whenever $am\in N$,
then either $sa\in\sqrt{(N:_{R}M)}$ or $sm\in N$ for each $a\in R$ and $m\in
M$. More recently, Ansari-Toroghy and Pourmortazavi, \cite{S-primary} studied
many more properties of this class of submodules.

Motivated and inspired by the above works, the purpose of this article is to
introduce a generalization of $S$-primary submodules to the context of weakly
$S$-primary submodules. A submodule $N$ of $M$ satisfying $(N:_{R}M)\cap
S=\phi$ is called a weakly $S$-primary submodule if there exists $s\in S$ such
that for $a\in R$ and $m\in M$, whenever $0\neq am\in N$, then either
$sa\in\sqrt{(N:_{R}M)}$ or $sm\in N$.

In section 2, many examples and characterizations of weakly $S$-primary
submodules are introduced (see for example, Example \ref{e1}, Theorem
\ref{char}, Theorem \ref{fm}). Moreover, several properties of weakly
$S$-primary submodules are obtained (see for example, Theorem \ref{(N:M)},
Propositions \ref{(I:s)}, \ref{p1}). We also investigate the behavior of this
structure under module homomorphisms, localizations, quotient modules and
Cartesian product of modules (see Propositions \ref{loc}, \ref{f}, Theorem
\ref{cart} and Corollary \ref{quot}).

Let $R$ be a ring and $M$ be an $R$-module. The idealization ring $R\ltimes M$
of $M$ in $R$ is defined as the set $\left\{  (r,m):r\in R,m\in M\right\}  $
with the usual componentwise addition and multiplication defined as
$(r,m)(s,n)=(rs,rn+sm)$. It can be easily verified that $R\ltimes M$ is a
commutative ring with identity $(1_{R},0_{M})$. If $I$ is an ideal of $R$ and
$N$ is a submodule of $M$, then $I\ltimes N=\left\{  (r,m):r\in I,m\in
N\right\}  $ is an ideal of $R\ltimes M$ if and only if $IM\subseteq N$. In
this case, $I\ltimes N$ is called a homogeneous ideal of $R\ltimes M$, see
\cite{AnWi}. For an $R$-module $M$, we start section 3 by clarifying the
relation between weakly $S$-primary submodules of $M$ and weakly
$S(M)$-primary ideals of the idealization ring $R(M)$ (Theorem \ref{Ideal}).
Let $f:R_{1}\rightarrow R_{2}$ be a ring homomorphism, $J$ be an ideal of
$R_{2}$, $M_{1}$ be an $R_{1}$-module, $M_{2}$ be an $R_{2}$-module (which is
an $R_{1}$-module induced naturally by $f$) and $\varphi:M_{1}\rightarrow
M_{2}$ be an $R_{1}$-module homomorphism. We conclude section 3 by
investigating some kinds of weakly $S$-primary submodules in the the
amalgamation $(R_{1}\Join^{f}J)$-module $M_{1}\Join^{\varphi}JM_{2}$ of
$M_{1}$ and $M_{2}$ along $J$ with respect to $\varphi$ (see Theorems
\ref{Amalg}, \ref{Amalg2}). Furthermore, we conclude some particular results
for the duplication of a module along an ideal (see Corollaries \ref{ca1},
\ref{ca2}, \ref{Dup}, \ref{Dup1} \ref{Dup2}).

As usual, $%
\mathbb{Z}
$, $%
\mathbb{Z}
_{n}$ and $%
\mathbb{Q}
$ denotes the ring of integers, the ring of integers modulo $n$ and the field
of rational numbers, respectively. For more details and terminology, one may
refer to \cite{Ali}, \cite{Majed}, \cite{At}, \cite{Gilmer}, \cite{Larsen}.

\section{Weakly $S$-primary Submodules}

\begin{definition}
Let $S$ be a multiplicatively closed subset of a ring $R$ and $N$ be a
submodule of an $R$-module $M$ with $(N:_{R}M)\cap S=\emptyset$. We call $N$ a
weakly $S$-primary submodule if there exists (a fixed) $s\in S$ such that for
$a\in R$ and $m\in M$, whenever $0\neq am\in N$ then either $sa\in
\sqrt{(N:_{R}M)}$ or $sm\in N$. The fixed element $s\in S$ is said to be a
weakly $S$-element of $N.$
\end{definition}

\begin{remark}
\label{r1}Let $S$ be a multiplicatively closed subset of a ring $R$.
\end{remark}

\begin{enumerate}
\item An ideal $I$ of $R$ is a weakly $S$-primary ideal if and only if $I$ is
a weakly $S$-primary submodule of the $R$-module $R$.

\item Any $S$-primary submodule is a weakly $S$-primary submodule.

\item Any weakly primary submodule $N$ of an $R$-module $M$ satisfying
$(N:_{R}M)\cap S=\emptyset$ is a weakly $S$-primary submodule of $M$.
Moreover, the two concepts coincide if $S\subseteq U(R)$ where $U(R)$ denotes
the set of units in $R$.
\end{enumerate}

\begin{example}
\label{e1}
\end{example}

\begin{enumerate}
\item Unlike the case of weakly primary submodules, the zero submodule need
not be weakly $S$-primary. For example, if we consider the multiplicatively
closed subset $S=\left\{  4^{m}:m\in%
\mathbb{N}
\cup\left\{  0\right\}  \right\}  $ of $%
\mathbb{Z}
$, then $\bar{0}$ is not weakly $S$-primary in the $%
\mathbb{Z}
$-module $%
\mathbb{Z}
_{4}$ as clearly $(0:_{%
\mathbb{Z}
}%
\mathbb{Z}
_{4})\cap S\neq\emptyset$.

\item For any multiplicatively closed subset $S$ of a ring $R$ and an
$R$-module $M$, if $(0:_{R}M)\cap S=\emptyset$, then $0$ (by definition) is a
weakly $S$-primary submodule of any $R$-module. Thus, the converse of (2) in
Remark \ref{r1} is not true in general. For a non trivial example consider the
$%
\mathbb{Z}
$-module $M=%
\mathbb{Z}
\ltimes\left(
\mathbb{Z}
_{2}\times%
\mathbb{Z}
_{2}\right)  $ and the submodule $N=0\ltimes\left\langle (\bar{1},\bar
{0})\right\rangle $ of $M$. For $S=\left\{  3^{m}:m\in%
\mathbb{N}
\cup\left\{  0\right\}  \right\}  $, we observe that $N$ is not an $S$-primary
submodule since for example, $2.(0,(\bar{1},\bar{1}))=(0,(\bar{0},\bar{0}))\in
N$ but for every $s\in S$, $2s\notin\sqrt{(N:_{R}M)}=\left\{  0\right\}  $ and
$s.(0,(\bar{1},\bar{1}))\notin N$. On the other hand, $N$ is weakly
$S$-primary in $M$. Indeed, choose $s=1$ and let $r_{1}\in R$, $(r_{2}%
,(\bar{a},\bar{b}))\in M$ such that $(0,(\bar{0},\bar{0}))\neq r_{1}%
.(r_{2},(\bar{a},\bar{b}))\in N$ and $sr_{1}\notin\sqrt{(N:_{R}M)}$. Then
$r_{1}\neq0$ and $(r_{1}r_{2},r_{1}.(\bar{a},\bar{b}))\in N$. It follows that
$r_{2}=0$ and $(\bar{0},\bar{0})\neq r_{1}.(\bar{a},\bar{b})\in\left\langle
(\bar{1},\bar{0})\right\rangle $. If $s(\bar{a},\bar{b})=(\bar{1},\bar{1})$ or
$(\bar{0},\bar{1})$, then $r_{1}.(\bar{a},\bar{b})\in\left\langle (\bar
{1},\bar{0})\right\rangle $ if and only if $r_{1}\in\left\langle
2\right\rangle $ and so $r_{1}.(\bar{a},\bar{b})=(\bar{0},\bar{0})$, a
contradiction. Thus, $s(\bar{a},\bar{b})\in\left\langle (\bar{1},\bar
{0})\right\rangle $ and $N$ is a weakly $S$-primary submodule of $M$.

\item While clearly every weakly $S$-prime submodule weakly $S$-primary, the
converse need not be true. For example, consider the $%
\mathbb{Z}
$-module $M=%
\mathbb{Z}
\times%
\mathbb{Z}
$ and let $S=\left\{  3^{m}:m\in%
\mathbb{N}
\cup\left\{  0\right\}  \right\}  $. Then $N=4%
\mathbb{Z}
\times%
\mathbb{Z}
$ is a (weakly) $S$-primary submodule of $M$. Indeed choose $s=1$ and let
$r\in%
\mathbb{Z}
$, $(a,b)\in%
\mathbb{Z}
\times%
\mathbb{Z}
$ such that $(0,0)\neq r.(a,b)\in N$ and $sr\notin\sqrt{(N:_{R}M)}%
=\left\langle 2\right\rangle $. Then clearly $a\in4%
\mathbb{Z}
$ and so $s(a,b)\in N$ as needed. On the other hand, $N$ is not weakly
$S$-prime since $2.(2,1)\in N$ but $2s\notin(N:_{R}M)=\left\langle
4\right\rangle $ and $s(2,1)\notin N$ for every $s\in S$.

\item Let $p$, $q$ be distinct prime integers and $k<n,$ $t<m$ be positive
integers. Consider $S=\{q^{n}:n\in%
\mathbb{N}
\}$ and the submodule $N=(\overline{p}^{k}\overline{q}^{t})$ of the $%
\mathbb{Z}
$-module $%
\mathbb{Z}
_{p^{n}q^{m}}.$ Then $N$ is a weakly $S$-primary submodule of $M$ associated
with $s=q^{t}\in S.$
\end{enumerate}

Let $N$ be a submodule of an $R$-module $M$ and $A$ be a subset of $R$. The
residual of $N$ by $A$ is the set $(N:_{M}A)=\{m\in M:Am\subseteq N\}$which is
a submodule of $M$ containing $N$.

\begin{theorem}
\label{char}Let $S$ be a multiplicatively closed subset of a ring $R$ and $N$
be a submodule of an $R$-module $M$ with $(N:_{R}M)\cap S=\emptyset$. Then the
following are equivalent.
\end{theorem}

\begin{enumerate}
\item $N$ is a weakly $S$-primary submodule of $M$.

\item There exists $s\in S$ such that for all $a\notin(\sqrt{(N:_{R}M)}:s)$,
$(N:_{M}a)\subseteq(0:_{M}a)\cup(N:_{M}s)$ $.$

\item There exists $s\in S$ such that for all $a\notin(\sqrt{(N:_{R}M)}:s)$,
$(N:_{M}a)=(0:_{M}a)$ or $(N:_{M}a)\subseteq(N:_{M}s)$.

\item There exists $s\in S$ such that for any $a\in R$ and for any submodule
$K$ of $M$, if $0\neq aK\subseteq N$, then $sa\in\sqrt{(N:_{R}M)}$ or
$sK\subseteq N.$

\item There exists $s\in S$ such that for any ideal $I$ of $R$ and a submodule
$K$ of $M$, if $0\neq IK\subseteq N$, then $sI\subseteq\sqrt{(N:_{R}M)}$ or
$sK\subseteq N.$
\end{enumerate}

\begin{proof}
(1)$\Rightarrow$(2). Suppose that $s\in S$ is a weakly $S$-element of $N$ and
let $a\notin(\sqrt{(N:_{R}M)}:s)$. Let $m\in(N:_{M}a)$. If $am=0$, then
$m\in(0:_{M}a)$. If $0\neq am\in N$, then $sa\notin\sqrt{(N:_{R}M)}$ implies
$sm\in N.$ Thus, $m\in(N:_{M}s)$ and so $(N:_{M}a)\subseteq(0:_{M}%
a)\cup(N:_{M}s)$.

(2)$\Rightarrow$(3). It is immediate.

(3)$\Rightarrow$(4). Let $s\in S$ be an element satisfying statement (3) and
suppose $0\neq aK\subseteq N$ and $sa\notin\sqrt{(N:_{R}M)}$ for some $a\in R$
and a submodule $K$ of $M$. Then $K\subseteq(N:_{M}a)\backslash(0:_{M}a)$ and
by (3) we get $K\subseteq(N:_{M}a)\subseteq(N:_{M}s)$. Thus, $sK\subseteq N$,
as needed.

(4)$\Rightarrow$(5). Choose $s\in S$ as in (4). Let $I$ be an ideal of $R$ and
$K$ a submodule of $M$ with $0\neq IK\subseteq N.$ Assume that $sI\nsubseteq
\sqrt{(N:_{R}M)}.$ Then there exists $a\in I$ with $sa\notin\sqrt{(N:_{R}M)}$.
If $aK\neq0$, then we have $sK\subseteq N$ by our assumption (4). Now, suppose
that $aK=0$. Since $IK\neq0$, there is some $b\in I$ with $bK\neq0$. If
$sb\notin\sqrt{(N:_{R}M)},$ then we have $sK\subseteq N$ by (4). If
$sb\in\sqrt{(N:_{R}M)},$ then as $sa\notin\sqrt{(N:_{R}M)},$ we conclude
$s(a+b)\notin\sqrt{(N:_{R}M)}$. Thus, $0\neq(a+b)K\subseteq N$ implies
$sK\subseteq N$ again by (4), as required.

(5)$\Rightarrow$(1). Take $I=aR$ and $K=Rm$ in (5).
\end{proof}

Next, we give a characterization for weakly $S$-primary submodule of faithful
multiplication modules.

\begin{theorem}
\label{fm}Let $S$ be a multiplicatively closed subset of a ring $R$ and $M$ be
a faithful multiplication $R$-module. Then $N$ is a weakly $S$-primary
submodule of $M$ if and only if $(N:_{R}M)\cap S=\emptyset$ and there exists
$s\in S$ such that whenever $K,$ $L$ are submodules of $M$ and $0\neq
KL\subseteq N$, then $sK\subseteq M-rad(N)$ or $sL\subseteq N$.
\end{theorem}

\begin{proof}
$\Longrightarrow)$ Let $s$ be a weakly $S$-element of $N$ and Suppose that
$0\neq KL\subseteq N$ for some submodules $K,$ $L$ of $M$. Since $M$ is
multiplication, we may write $K=IM$ for some ideal $I$ of $R$. Hence, $0\neq
IL\subseteq N$ and so by Theorem \ref{char}, $sI\subseteq\sqrt{(N:_{R}M)}$ or
$sL\subseteq N$ . It follows that $sK=sIM\subseteq\sqrt{(N:_{R}M)}M=M-rad(N)$
or $sL\subseteq N$.

$\Longleftarrow)$ Let $s\in S$ such that whenever $K,$ $L$ are submodules of
$M$ and $0\neq KL\subseteq N$, then $sK\subseteq M-rad(N)$ or $sL\subseteq N$.
Suppose that $0\neq IL\subseteq N$ for some ideal $I$ of $R$ and submodule $L$
of $M$ and $sL\nsubseteq N$. Then as $M$ is faithful, we have $0\neq
(IM)L\subseteq N$. By assumption, we have $sIM\subseteq M-rad(N)$ and so
$sI\subseteq(M-rad(N):M)=\sqrt{(N:_{R}M)}$. Thus, $s$ is a weakly $S$-element
of $N$ and by Theorem \ref{char}, we are done .
\end{proof}

Let $I$ be a proper ideal of a ring $R$. In the following proposition, the
notation $Z_{I}(R)$ denotes the set $\{r\in R:rs\in I$ for some $s\in
R\backslash I\}$.

\begin{theorem}
\label{(N:M)}Let $S$ be a multiplicatively closed subset of a ring $R$ and $N$
be a submodule of an $R$-module $M$. The following statements hold.
\end{theorem}

\begin{enumerate}
\item If $N$ is a weakly $S$-primary submodule of $M$, then for every
submodule $K$ with $(N:_{R}K)\cap S=\emptyset$ and $Ann(K)=0,$ $(N:_{R}K)$ is
a weakly $S$-primary ideal of $R$. In particular, if $M$ is faithful, then
$(N:_{R}M)$ is a weakly $S$-primary ideal of $R.$

\item If $M$ is multiplication and $(N:_{R}M)$ is a weakly $S$-primary ideal
of $R$, then $N$ is a weakly $S$-primary submodule of $M$.

\item Let $M$ be a faithful multiplication module and $I$ be an ideal of $R$.
Then $I$ is weakly $S$-primary ideal of $R$ if and only if $IM$ is a weakly
$S$-primary submodule of $M$.

\item If $N$ is a weakly $S$-primary submodule of $M$ and $T$ is a subset of
$R$ such that $(0:_{M}T)=0$ and $Z_{(N:_{R}M)}(R)\cap T=\emptyset$, then
$(N:_{M}T)$ is a weakly $S$-primary submodule of $M$.
\end{enumerate}

\begin{proof}
(1) Let $s$ be a weakly $S$-element of $N$ and suppose that $0\neq
ab\in(N:_{R}K)$ for some $a,b\in R$. Then, $0\neq abK\subseteq N$ as
$Ann(K)=0$ and Theorem \ref{char} implies $sa\in\sqrt{(N:_{R}M)}\subseteq
\sqrt{(N:_{R}K)}$ or $sbK\subseteq N$. Thus, $sa\in\sqrt{(N:_{R}K)}$ or
$sb\in(N:_{R}K)$, as needed.

(2) Let $(N:_{R}M)$ be a weakly $S$-primary ideal of $R$ associated to $s\in
S$. Suppose that $0\neq IK\subseteq N$ for some ideal $I$ of $R$ and a
submodule $K$ of $M$. We may write $K:=JM$ for some ideal $J$ of $R$ as $M$ is
multiplication. Now, $0\neq IJ\subseteq(N:_{R}M)$ and so by \cite{WS-primary},
$sI\subseteq\sqrt{(N:_{R}M)}$ or $sJ\subseteq(N:_{R}M)$. Thus, $sI\subseteq
\sqrt{(N:_{R}M)}$ or $sK=sJM\subseteq(N:_{R}M)M=N$. Therefore, $N$ is a weakly
$S$-primary submodule of $M$ by Theorem \ref{char}.

(3) Put $N=IM$ in (1) and (2). Since $(N:_{R}M)=(IM:_{R}M)=I,$ the claim is clear.

(4) First, we show that $((N:_{M}T):_{R}M)\cap S=\emptyset$. Assume
$s\in((N:_{M}T):_{R}M)\cap S$, then $sT\subseteq(N:_{R}M)$ and since
$Z_{(N:_{R}M)}(R)\cap T=\emptyset$, we conclude that $s\in(N:_{R}M)$ which
contradicts $(N:_{R}M)\cap S=\emptyset$. Let $s$ be a weakly $S$-element of
$N$ and suppose that $0\neq am\in(N:_{M}T)$ for some $a\in R$ and $m\in M$.
Since $(0:_{M}T)=0$, we have $0\neq a(Tm)\subseteq N$ which yields $sa\in
\sqrt{(N:_{R}M)}\subseteq\sqrt{((N:_{M}T):_{R}M)}$ or $sTm\subseteq N$. Thus,
$sa\in$ $\sqrt{((N:_{M}T):_{R}M)}$ or $sm\in(N:_{M}T)$ and $(N:_{M}T)$ is
weakly $S$-primary in $M$.
\end{proof}

We show by the following example that the condition "faithful module" in (1)
of Theorem \ref{(N:M)} is crucial.

\begin{example}
Consider the multiplicatively closed subset $S=\left\{  3^{m}:m\in%
\mathbb{N}
\cup\left\{  0\right\}  \right\}  $ of $%
\mathbb{Z}
$ and the $%
\mathbb{Z}
$-module $M=%
\mathbb{Z}
_{10}\times%
\mathbb{Z}
_{10}$. Then $N=\overline{0}\times\overline{0}$ is a weakly $S$-primary
submodule of $M$ but $(N:_{%
\mathbb{Z}
}M)=\left\langle 10\right\rangle $ is not a weakly $S$-primary ideal of $%
\mathbb{Z}
$.
\end{example}

In view of Theorem \ref{(N:M)}, we have the following equivalent statements.

\begin{corollary}
\label{IM}Let $M$ be a faithful multiplication $R$-module and $N$ be a
submodule of $M$. The following are equivalent.
\end{corollary}

\begin{enumerate}
\item $N$ is a weakly $S$-primary submodule of $M$.

\item $(N:_{R}M)$ is a weakly $S$-primary ideal of $R.$

\item $N=IM$ for some weakly $S$-primary ideal $I$ of $R$.
\end{enumerate}

Let $N$ be a proper submodule of an $R$-module $M$. Then $N$ is said to be a
maximal weakly $S$-primary submodule if there is no weakly $S$-primary
submodule which contains $N$ properly. In the following corollary, by $Z(M)$,
we denote the set $\{r\in R:rm=0$ for some $m\in M\backslash\{0_{M}\}\}$.

\begin{corollary}
Let $N$ be a submodule of of an $R$-module $M$ such that $Z_{(N:_{R}M)}(R)\cup
Z(M)\subseteq\sqrt{(N:_{R}M)}$.\ If $N$ is a maximal weakly $S$-primary
submodule of $M$, then $N$ is an $S$-primary submodule of $M$.
\end{corollary}

\begin{proof}
Let $s\in S$ be a weakly $S$-element of $N$ and let $a\in R$, $m\in M$ such
that $am\in N$ and $sa\notin\sqrt{(N:_{R}M)}.$ Since $a\notin\sqrt{(N:_{R}M)}%
$, then by assumption, $a\notin Z_{(N:_{R}M)}(R)$ and $(0:_{M}a)=0$.
Therefore, by Theorem \ref{(N:M)}(5), $(N:_{M}a)$ is a weakly $S$-primary
submodule of $M$. Since $N$ is a maximal weakly $S$-primary submodule, we
conclude $sm\in(N:_{M}a)=N$ as needed.
\end{proof}

\begin{proposition}
\label{(I:s)}Let $S$ be a multiplicatively closed subset of a ring $R$ and $N$
be a submodule of an $R$-module $M$ such that $(N:_{R}M)\cap S=\emptyset$. Then
\end{proposition}

\begin{enumerate}
\item If $(N:_{M}s)$ is a weakly primary submodule of $M$ for some $s\in S$,
then $N$ is a weakly $S$-primary submodule of $M$.

\item If $N$ is a non-zero weakly $S$-primary submodule of $M$ and $S\cap
Z(M)=\emptyset,$ then $(N:_{M}s)$ is a weakly primary submodule of $M$ for
some $s\in S.$
\end{enumerate}

\begin{proof}
(1) Choose $s\in S$ such that $(N:_{M}s)$ is a weakly primary submodule of
$M$. Suppose that $0\neq am\in N\subseteq(N:_{M}s)\ $for some $a\in R,$ $m\in
M.$ Since $(N:_{M}s)$ is weakly primary, we have either $a\in\sqrt
{((N:_{M}s):_{R}M)}=\sqrt{((N:_{R}M):_{R}s)}$ or $m\in(N:_{M}s)$. Thus,
$sa\in\sqrt{(N:_{R}M)}$ or $sm\in N$ and\ so $N$ is a weakly $S$-primary
submodule of $M$.

(2) Let $s$ be a weakly $S$-element of $N$ and $a\in R$, $m\in M$ with $0\neq
am\in(N:_{M}s)$. Since $S\cap Z(M)=\emptyset$, clearly we have $0\neq sam\in
N$ which implies either $s^{2}a\in\sqrt{(N:_{R}M)}$ or $sm\in N$. Thus,
$sa\in\sqrt{(N:_{R}M)}\subseteq\sqrt{((N:_{M}s):_{R}M)}$ or $m\in(N:_{M}s)$,
as needed.
\end{proof}

If $S\cap Z(M)\neq\emptyset$, then the converse of Proposition \ref{(I:s)}
need not be true as we can see in the following example.

\begin{example}
\label{ex11}Consider $M=%
\mathbb{Z}
\times%
\mathbb{Z}
_{pq}$ as a $%
\mathbb{Z}
$-module and let $S=\left\{  p^{n}:n\in%
\mathbb{N}
\right\}  .$ Here, observe that $S\cap Z(M)=S\neq\emptyset$. Now,
$N=\left\langle 0\right\rangle \times\left\langle \bar{0}\right\rangle $ is a
weakly $S$-primary submodule of $M$. On the other hand, for each positive
integer $n$, $(N:_{M}p^{n})\mathbf{=}\left\langle 0\right\rangle
\mathbf{\times}\left\langle \overline{q}\right\rangle $ is not a weakly
primary submodule of $M$. Indeed, $q.(0,\bar{1})\in(N:_{M}p^{n})$ but neither
$q\in\sqrt{((N:_{M}p^{n}):_{R}M)}=\left\langle 0\right\rangle $ nor
$(0,\bar{1})\in(N:_{M}p^{n})$.
\end{example}

\begin{proposition}
\label{p1}Let $S$ be a multiplicatively closed subset of a ring $R$ and $N$ be
a weakly $S$-primary submodule of a\ finitely generated faithful
multiplication $R$-module $M$.
\end{proposition}

\begin{enumerate}
\item If $\{0_{R}\}$ is an $S$-primary ideal of $R$, then $M$-$rad(N)$ is an
$S$-prime submodule of $M.$

\item If $N$ is not $S$-primary, then $N^{2}=0_{M}$ and $M-rad(N)=M-rad(0_{M}%
).$ Moreover, nonzero weakly $S$-primary submodules and $S$-primary submodules
coincide if $R$ is a reduced ring.
\end{enumerate}

\begin{proof}
(1) Suppose that $N$ is a weakly $S$-primary submodule of $M.$ Then
$(N:_{R}M)$ is a weakly $S$-primary ideal of $R$ by Corollary \ref{IM}.
Therefore, $\sqrt{(N:_{R}M)}$ is an $S$-prime ideal of $R$, \cite{WS-primary}%
$.$Thus, $M$-$rad(N)=\sqrt{(N:_{R}M)}M$ is an $S$-prime submodule of $M$ by
\cite[Proposition 2.9 (ii)]{S-prime subm}$.$

(2) Suppose $N$ is not $S$-primary. Then $(N:_{R}M)$ is a weakly $S$-primary
ideal of $R$ that is not $S$-primary by Corollary \ref{IM}. It follows by
\cite[Proposition 2(1)]{WS-primary} that $(N:_{R}M)^{2}=0_{R}$. Since $M$ is
multiplication, we have $N^{2}=(N:_{R}M)^{2}M=0_{M}$ and $M-rad(N)=\sqrt
{(N:_{R}M)}M=\sqrt{0_{R}}M\subseteq\sqrt{0_{M}:M}M=M-rad(0_{M}).$ The rest of
the proof is straightforward.
\end{proof}

\begin{proposition}
\label{loc}Let $N$ be a submodule of an $R$-module $M$ and $S$ be a
multiplicatively closed subset of $R$.

\begin{enumerate}
\item If $N$ is a weakly $S$-primary submodule of $M$, then $S^{-1}N$ is a
weakly primary submodule of $S^{-1}M$. Moreover, if $Z(M)\cap S=\emptyset$,
then there exists an $s\in S$ such that $(N:_{M}t)\subseteq(N:_{M}s)$ for all
$t\in S$.

\item If $M$ is finitely generated and $Z(M)\cap S=\emptyset$, then the
converse of (1) holds.
\end{enumerate}
\end{proposition}

\begin{proof}
(1) Let $s\in S$ be a weakly $S$-element of $N$ and $\frac{a}{s_{1}}\in
S^{-1}R$, $\frac{m}{s_{2}}\in S^{-1}M$ such that $0_{S^{-1}M}\neq\frac
{a}{s_{1}}\frac{m}{s_{2}}\in S^{-1}N$. Then $uam\in N$ for some $u\in S$. If
$uam=0$, then $\frac{am}{s_{1}s_{2}}=\frac{uam}{us_{1}s_{2}}=0_{S^{-1}M}$, a
contradiction. Thus, $0\neq uam\in N$ which implies that either $sua\in
\sqrt{(N:_{R}M)}$ or $sm\in N.$ Thus, $\frac{a}{s_{1}}=\frac{sua}{sus_{1}}\in
S^{-1}\sqrt{(N:_{R}M)}\subseteq\sqrt{(S^{-1}N:_{S^{-1}R}S^{-1}M)}$ or
$\frac{m}{s_{2}}=\frac{sm}{ss_{2}}\in S^{-1}N$, as needed. For the rest of the
proof, let $t\in S$ and $0\neq m\in(N:_{M}t).$ Then $0\neq tm\in N$ as
$Z(M)\cap S=\emptyset$ and so $sm\in N$ as $st\in(N:_{M}M)\cap S$ gives a
contradiction. Therefore, $m\in(N:_{M}s)$ and so $(N:_{M}t)\subseteq(N:_{M}s)$
for all $t\in S$.

(2) Suppose that $M$ is finitely generated, $S^{-1}N$ is a weakly primary
submodule of $S^{-1}M$ and there is a fixed $s\in S$ such that $(N:_{M}%
t)\subseteq(N:_{M}s)$ for all $t\in S$. Since $S^{-1}N$ is proper, we have
$(N:_{R}M)\cap S=\emptyset$. Let $0\neq am\in N$ for some $a\in R$ and $m\in
M$. Then $0\neq\frac{a}{1}\frac{m}{1}\in S^{-1}N$ as $Z(M)\cap S=\emptyset.$
Since $S^{-1}N$ is weakly primary, we have either $\frac{a}{1}\in\sqrt
{(S^{-1}N:_{S^{-1}R}S^{-1}M)}=S^{-1}\sqrt{(N:_{R}M)}$ as $M$ is finitely
generated or $\frac{m}{1}\in S^{-1}N.$ Hence, $s_{1}a\in\sqrt{(N:_{R}M)}$ for
some $s_{1}\in S$ or $s_{2}m\in N$ for some $s_{2}\in S$. If $s_{1}a\in
\sqrt{(N:_{R}M)}$, then $s_{1}^{n}a^{n}\in(N:_{R}M)$ for some positive integer
$n$ and $a^{n}\in((N:_{R}M):s_{1}^{n})=((N:_{M}s_{1}^{n}):_{R}M)\subseteq
((N:_{M}s):_{R}M)$ by our assumption. Thus $sa^{n}\in(N:_{R}M)$ and
$sa\in\sqrt{(N:_{R}M)}.$ If $s_{2}m\in N$, then we conclude $m\in(N:_{M}%
s_{2})\subseteq(N:_{M}s)$ and so $sm\in N$. Consequently, $N$ is a weakly
$S$-primary submodule of $M.$
\end{proof}

Let $M$ be an $R$-module and $S\subseteq S^{\prime}$ be two multiplicatively
closed subsets of $R$. If $N$ is a weakly $S$-primary submodule of $M$ and
$(N:_{R}M)\cap S^{\prime}=\emptyset$, then it is clear that $N$ is a weakly
$S^{\prime}$-primary submodule of $M.$ In \cite{Gilmer}, the saturation of $S$
is defined as the multiplicatively closed subset $S^{\ast}=\{x\in R:xy\in S$
for some $y\in R$\} which contains $S$.

\begin{proposition}
Let $S$ be a multiplicatively closed subset of a ring $R$ and $N$ be a
submodule of an $R$-module $M$. Then $N$ is weakly $S$-primary if and only if
$N$ is weakly $S^{\ast}$-primary.
\end{proposition}

\begin{proof}
Suppose $N$ is weakly $S^{\ast}$-primary in $M$ associated to $s^{\ast}\in
S^{\ast}$. Note that $(N:_{R}M)\cap S=\emptyset$ as $S\subseteq S^{\ast}$.
Choose $s=s^{\ast}y\in S$ for some $y\in R$ and let $0\neq am\in N$ for some
$a\in R$ and $m\in M$. Then either $s^{\ast}a\in\sqrt{(N:_{R}M)}$ or $s^{\ast
}m\in N$. Hence, $sa\in\sqrt{(N:_{R}M)}$ or $sm\in N$ and $N$ is weakly
$S$-primary submodule of $M$. Conversely, suppose that $N$ is weakly
$S$-primary. We need to prove that $(N:_{R}M)\cap S^{\ast}=\emptyset$. If
$s^{\ast}\in$ $(N:_{R}M)\cap S^{\ast}$, then there is $y\in R$ such that
$s=s^{\ast}y\in(N:_{R}M)\cap S$ which is a contradiction$.$ Thus, $N$ is
weakly $S^{\ast}$-primary as $S\subseteq S^{\ast}.$
\end{proof}

\begin{proposition}
\label{f}Let $M$ and $M^{\prime}$ be two $R$-modules and $f:M\rightarrow
M^{\prime}$\ be a homomorphism. For a multiplicatively closed subset $S$ of
$R$, we have:
\end{proposition}

\begin{enumerate}
\item If $f$ is an epimorphism and $N$ is a weakly $S$-primary submodule of
$M$ containing $Ker(f)$, then $f(N)$ is a weakly $S$-primary submodule of
$M^{\prime}.$

\item If $f$ is a monomorphism and $N^{\prime}$ is a weakly $S$-primary
submodule of $M^{\prime}$, then $f^{-1}(N^{\prime})$ is a weakly $S$-primary
submodule of $M.$
\end{enumerate}

\begin{proof}
(1) Let $s\in S$ be a weakly $S$-element of $N$. First, as $Ker(f)\subseteq
N$, it follows that $(f(N):_{R}M^{\prime})\cap S=\emptyset$. Suppose that
$0\neq am^{\prime}\in f(N)$ for some $a\in R$ and $m^{\prime}\in M^{\prime}$.
Choose $m\in M$ with $m^{\prime}=f(m)$. Then $0\neq af(m)=f(am)\in f(N)$ and
since $Ker(f)\subseteq N,$ we have $0\neq am\in N$. It follows that either
$sa\in\sqrt{(N:_{R}M)}$ or $sm\in N$. Thus, clearly we have either $sa\in
\sqrt{(f(N):_{R}M^{\prime})}$ or $sm^{\prime}=f(sm)\in f(N)$ and $f(N)$ is a
weakly $S$-primary submodule of $M^{\prime}.$

(2) Let $s\in S$ be a weakly $S$-element of $N^{\prime}$ and note that clearly
$(f^{-1}(N^{\prime}):_{R}M)\cap S=\emptyset$. Let $a\in R$ and $m\in M$ such
that $0\neq am\in f^{-1}(N^{\prime})$. Then $0\neq f(am)=af(m)\in N^{\prime}$
as $f$ is a monomorphism. It follows either $sa\in\sqrt{(N^{\prime}%
:_{R}M^{\prime})}$ or $sf(m)\in N^{\prime}$. Thus, we conclude either
$sa\in\sqrt{(f^{-1}(N^{\prime}):_{R}M})$ or $sm\in f^{-1}(N^{\prime})$ and we
are done.
\end{proof}

\begin{corollary}
\label{quot}Let $S$ be a multiplicatively closed subset of a ring $R$ and
$K\subseteq N$ be two submodules of an $R$-module $M$.
\end{corollary}

\begin{enumerate}
\item If $N$ is a weakly $S$-primary submodule of $M$, then $N/K$ is a weakly
$S$-primary submodule of $M/K$.

\item If $K^{\prime}$ is a weakly $S$-primary submodule of $M$, then
$K^{\prime}\cap N$ is a weakly $S$-primary submodule of $N.$

\item If $N/K$ is a weakly $S$-primary submodule of $M/K$ and $K$ is a
(weakly) $S$-primary submodule of $M$, then $N$ is a (weakly) $S$-primary
submodule of $M$.
\end{enumerate}

\begin{proof}
Observe that $(N/K:_{R}M/K)\cap S=\emptyset$ if and only if $(N:_{R}M)\cap
S=\emptyset$.

(1). The claim follows by Proposition \ref{f}(1) considering the canonical
epimorphism $\pi:M\rightarrow M/K$ defined by $\pi(m)=m+K$.

(2). This follows by Proposition \ref{f}(2) considering the natural injection
$i:N\rightarrow M$ defined by $i(m)=m$ for all $m\in N$.

(3). Let $s\in S$ be a weakly $S$-element of $N/K$ and $s\prime\in S$ be a
(weakly) $S$-element of $K$. Let $a\in R$ and $m\in M$ such that $am\in N.$ If
$am\in K$, then either $s\prime a\in\sqrt{(K:_{R}M)}\subseteq\sqrt{(N:_{R}M)}$
or $s^{\prime}m\in K\subseteq N$. If $am\notin K$, then $K\neq a(m+K)\in N/K$
which implies that either $sa\in\sqrt{(N/K:_{R}M/K)}$ or $s(m+K)\in N/K$.
Thus, $sa\in\sqrt{(N:_{R}M)}$ or $sm\in N$. It follows that $N$ is an
$S$-primary submodule of $M$ associated with $s=ss^{\prime}\in S.$
\end{proof}

The converse of Corollary \ref{quot}(1) does not hold in general. For
instance, consider the submodules $N=K=\left\langle p_{1}p_{2}\right\rangle $
of the $%
\mathbb{Z}
$-module $%
\mathbb{Z}
$ and the multiplicatively closed subset $S=\left\{  p_{3}^{n}:n\in%
\mathbb{N}
\cup\left\{  0\right\}  \right\}  $ of $%
\mathbb{Z}
$ where $p_{1},p_{2}$ and $p_{3}$ are distinct prime numbers. Then clearly
$N/K=0$ is a weakly $S$-primary submodule of $%
\mathbb{Z}
/K$ but $N$ is not a weakly $S$-primary submodule of $%
\mathbb{Z}
$ as $0\neq p_{1}\cdot p_{2}\in N$ but neither $sp_{1}\in\sqrt{(N:_{%
\mathbb{Z}
}%
\mathbb{Z}
)}=\left\langle p_{1}p_{2}\right\rangle $ nor $sp_{2}\in N$ for all $s\in S$.

\begin{proposition}
\label{int}Let $S$ be a multiplicatively closed subset of a ring $R$ and $N$
be a weakly $S$-primary submodule of an $R$-module $M$.
\end{proposition}

\begin{enumerate}
\item For any submodule $K$ of $M$ with $(K:_{R}M)M=K$ and $(K:_{R}M)\cap
S\neq\emptyset$, $N\cap K$ is a weakly $S$-primary submodule of $M$.

\item If $K$ is a weakly $S$-primary submodule of $M$ such that $((N+K):_{R}%
M)\cap S=\emptyset$, then $N+K$ is a weakly $S$-primary submodule of $M.$
\end{enumerate}

\begin{proof}
(1) Suppose that $0\neq am\in N\cap K\subseteq N$ for some $a\in R$ and $m\in
M$. Then there exists a $s\in S$ with either $sa\in\sqrt{(N:_{R}M)}$ or $sm\in
N$. Take $t\in(K:_{R}M)\cap S$. Then $sta\in\sqrt{(N:_{R}M)}\cap
(K:_{R}M)\subseteq\sqrt{(N\cap K:_{R}M)}$ or $stm\in N\cap(K:_{R}M)M=N\cap K$.
Thus, $N\cap K$ is a weakly $S$-primary submodule of $M$ associated with
$st\in S.$

(2) By Corollary \ref{quot}(1), $N/(N\cap K)$ is a weakly $S$-primary
submodule of $M/(N\cap K)$. Also, the isomorphism $N/(N\cap K)\cong(N+K)/K$
yields that $(N+K)/K$ is a weakly $S$-primary submodule of $M/K$. Now,
Corollary \ref{quot}(4) implies that $N+K$ is a weakly $S$-primary submodule
of $M$.
\end{proof}

We note that the condition $(K:_{R}M)\cap S\neq\emptyset$ in (1) of
Proposition \ref{int} can not be omitted. Indeed, if $N$ is weakly primary and
$K$ is as above, then $N\cap K$ need not be weakly primary. For example,
consider the $%
\mathbb{Z}
$-module $%
\mathbb{Z}
_{72}$, $S=\left\{  3^{n}:n\in%
\mathbb{N}
\right\}  $, $N=\left\langle \overline{4}\right\rangle $ and $K=\left\langle
\overline{9}\right\rangle $. Then $N\cap K=\left\langle \overline
{36}\right\rangle $ is not a weakly primary submodule of $%
\mathbb{Z}
_{72}$ but observe from Example \ref{e1}(2) that it is a weakly $S$-primary submodule.

\begin{theorem}
\label{cart}Let $S,$ $S^{\prime}$ be multiplicatively closed subsets of rings
$R$, $R^{\prime}$ respectively and $N$, $N^{\prime}$ be non-zero submodules of
an $R$-module $M$ and an $R^{\prime}$-module $M^{\prime}$, respectively.
Consider $M\times M^{\prime}$ as an $(R\times R^{\prime})$-module. Then the
following are equivalent.
\end{theorem}

\begin{enumerate}
\item $N\times N^{\prime}$ is a weakly $S\times S^{\prime}$-primary submodule
of $M\times M^{\prime}.$

\item $N$ is an $S$-primary submodule of $M$ and $(N^{\prime}:_{R^{\prime}%
}M^{\prime})\cap S^{\prime}\neq\emptyset$ or $N^{\prime}$ is an $S^{\prime}%
$-primary submodule of $M^{\prime}$ and $(N:_{R}M)\cap S\neq\emptyset$

\item $N\times N^{\prime}$ is an $S\times S^{\prime}$-primary submodule of
$M\times M^{\prime}.$
\end{enumerate}

\begin{proof}
(1)$\Rightarrow$(2). Choose a weakly $S\times S^{\prime}$-element
$(s,s^{\prime})$ of $N\times N^{\prime}$ and $0\neq$ $m\in N$ .

Case I: Suppose $(N:_{R}M)\cap S=\emptyset=(N^{\prime}:_{R^{\prime}}M^{\prime
})\cap S^{\prime}$. Then $(0,0)\neq(1,0)(m,1)\in N\times N^{\prime}$ and so
either $(s,s^{\prime})(1,0)\in\sqrt{(N\times N^{\prime}:_{R\times R^{\prime}%
}M\times M^{\prime})}=\sqrt{(N:_{R}M)}\times\sqrt{(N^{\prime}:_{R^{\prime}%
}M^{\prime})}$ or $(s,s^{\prime})(m,1)\in N\times N^{\prime}.$ Hence, we have
either $s^{n}\in(N:_{R}M)\cap S$ for some positive integer $n$ or $s^{\prime
}\in N^{\prime}\cap S^{\prime}\subseteq(N^{\prime}:_{R^{\prime}}M^{\prime
})\cap S^{\prime}$, a contradiction.

Case II. Assume that $(N:_{R}M)\cap S\neq\emptyset$. Suppose $am^{\prime}\in
N^{\prime}$ for some $a\in R^{\prime}$ and $m^{\prime}\in M^{\prime}$. Then
$(0,0)\neq(1,a)(m,m^{\prime})\in N\times N^{\prime}$ implies either
$(s,s^{\prime})(1,a)\in\sqrt{(N:_{R}M)}\times\sqrt{(N^{\prime}:_{R^{\prime}%
}M^{\prime})}$ or $(s,s^{\prime})(m,m^{\prime})\in N\times N^{\prime}$. Thus,
$s^{\prime}a\in\sqrt{(N^{\prime}:_{R^{\prime}}M^{\prime})}$ or $s^{\prime
}m^{\prime}\in N^{\prime}$ and $N^{\prime}$ is an $S^{\prime}$-primary
submodule of $M^{\prime}.$

Case III. Assume that $(N^{\prime}:_{R^{\prime}}M^{\prime})\cap S^{\prime}%
\neq\emptyset$. We can prove in a similar way that $N$ is an $S$-primary
submodule of $M$.

(2)$\Rightarrow$(3). see \cite[Theorem 2.20]{S-primary}.

(3)$\Rightarrow$(1). is immediate.
\end{proof}

In view of the above theorem, we conclude the following generalization.

\begin{theorem}
\label{cart2}Let $M=M_{1}\times M_{2}\times\cdots\times M_{n}$ be an
$R=R_{1}\times R_{2}\times\cdots\times R_{n}$-module and $S=S_{1}\times
S_{2}\times\cdots\times S_{n}$ where $R_{i}$ is a ring, $S_{i}$ is a
multiplicatively closed subset of $R_{i}$ and $N_{i}$ is a non-zero submodule
of $M_{i}$ for each $i=1,2,...,n$. Then the following assertions are equivalent.
\end{theorem}

\begin{enumerate}
\item $N=N_{1}\times N_{2}\times\cdots\times N_{n}$ is a weakly $S$-primary
submodule of $M.$

\item There exists $i\in\left\{  1,2,...,n\right\}  $ such that $N_{i}$ is an
$S_{i}$-primary submodule of $M_{i}$ and $(N_{j}:_{R_{j}}M_{j})\cap S_{j}%
\neq\emptyset$ for all $j\neq i.$

For $i=1,2,...,n$, $N_{i}$ is an $S$-primary submodule of $M_{i}$ and
$(N_{j}:_{R_{j}}M_{j})\cap S_{j}\neq\emptyset$ for all $j\neq i.$
\end{enumerate}

\begin{proof}
To prove the claim, we use the mathematical induction on $n$. For $n=2,$ see
Theorem \ref{cart}. Assume that the claim holds for all $k<n$. Suppose
$N=N_{1}\times N_{2}\times\cdots\times N_{n}$ is a weakly $S$-primary
submodule of $M$. Let $R^{\prime}=R_{1}\times R_{2}\times\cdots\times R_{n-1}%
$, $N^{\prime}=N_{1}\times N_{2}\times\cdots\times N_{n-1}$ and $S^{\prime
}=S_{1}\times S_{2}\times\cdots\times S_{n-1}$. By Theorem \ref{cart}, we have
either $N_{n}$ is weakly $S$-primary in $M_{n}$ and $(N^{\prime}:_{R^{\prime}%
}M^{\prime})\cap S^{\prime}\neq\emptyset$ or $N^{\prime}$ is a weakly
$S^{\prime}$-primary submodule of $M^{\prime}$ and $S_{n}\cap(N_{n}:_{R_{n}%
}M_{n})\neq\emptyset$. In the first case, we are done as clearly
$(N_{j}:_{R_{j}}M_{j})\cap S_{j}\neq\emptyset$ for all $j\neq n$. In the
second case, we conclude the result by the induction hypothesis. \ 
\end{proof}

Let $R$ be a ring, $M$ be an $R$-module and consider the idealization ring
$R\ltimes M$ of $M$ in $R$. It is proved in \cite[Theorem 3.2]{AnWi} that if
$I\ltimes N$ is a homogenous ideal in $R\ltimes M$, then $\sqrt{I\ltimes
N}=\sqrt{I}\ltimes M$. For a multiplicatively closed subset $S$ of $R$,
clearly $S\ltimes N=\{(s,n):$ $s\in S$, $n\in N\}$ is a multiplicatively
closed subset of $R\ltimes M$.

\begin{theorem}
\label{Ideal}Let $S$ be a multiplicatively closed subset of a ring $R,$ $I$ be
an ideal of $R$ and $K\subseteq N$ be submodules of an $R$-module $M$ with
$IM\subseteq N$. If $I\ltimes N$ is a weakly $S\ltimes K$-primary ideal of
$R\ltimes M$, then $I$ is a weakly $S$-primary ideal of $R$ and $N$ is a
weakly $S$-primary submodule of $M$ whenever $(N:_{R}M)\cap S=\emptyset$.
Furthermore, there exists $s\in S$ such that for all $a,b\in R$, $ab=0$,
$sa\notin\sqrt{I}$, $sb\notin I$ implies $a,b\in ann(N)$ and for all $c\in R$,
$m\in M$, $cm=0$, $sc\notin\sqrt{(N:_{R}M)}$, $sm\notin N$ implies $c\in
ann(I)$ and $m\in(0:_{M}I)$.
\end{theorem}

\begin{proof}
(1) It is clear that $(S\ltimes K)\cap(I\ltimes N)=\emptyset$ if and only if
$I\cap S=\emptyset$. Let $(s,k)$ be a weakly $S\ltimes K$-primary element of
$I\ltimes N$ and let $a,b\in R$ with $0\neq ab\in I$. Then $(0,0)\neq
(a,0)(b,0)\in I\ltimes N$ and so either $(s,k)(a,0)\in\sqrt{I\ltimes N}%
=\sqrt{I}\ltimes M$ or $(s,k)(b,0)\in I\ltimes N$. Hence, we have either
$sa\in\sqrt{I}$ or $sb\in I$ and $I$ is a weakly $S$-primary ideal of $R$. To
show that $N$ is weakly $S$-primary, let $0\neq am\in N$ for $a\in R$, $m\in
M$. Then $(0,0)\neq(a,0)(0,m)\in I\ltimes N$ and so $(sa,ak)=(s,k)(a,0)\in
\sqrt{I\ltimes N}=\sqrt{I}\ltimes M$ or $(0,sm)=(s,k)(0,m)\in I\ltimes N$.
Thus, we conclude either $sa\in\sqrt{I}\subseteq\sqrt{(N:_{R}M)}$ or $sm\in N$
and so $N$ is a weakly $S$-primary submodule of $M$. \ 

Now, let $a,b\in R$ such that $ab=0$, $sa\notin\sqrt{I}$ and $sb\notin I$.
Suppose $a\notin ann(N)$ so that there exists $n\in N$ such that $an\neq0$.
Then $(0,0)\neq(a,0)(b,n)=(0,an)\in I\ltimes N$ and so either $(s,k)(a,0)\in
\sqrt{I\ltimes N}$ or $(s,k)(b,n)\in I\ltimes N$. Hence, $sa\in\sqrt{I}$ or
$sb\in I$, a contradiction. Similarly, if $b\notin ann(N),$ then we get a
contradiction. Therefore, $a,b\in ann(N)$ as needed. Next, we assume for $c\in
R$, $m\in M$ that $cm=0$, $sc\notin\sqrt{(N:_{R}M)}$ and $sm\notin N$. Assume
on the contrary that $c\notin ann(I)$. Then there exists $a\in I$ such that
$ca\neq0$. Hence, $(0,0)\neq(c,0)(a,m)=(ca,0)\in I\ltimes N$ and so
$(s,k)(c,0)\in\sqrt{I\ltimes N}$ or $(s,k)(a,m)\in I\ltimes N$. Therefore,
$sc\in\sqrt{I}\subseteq\sqrt{(N:_{R}M)}$ or $sm+ka\in N$ which gives $sm\in N$
as $K\subseteq N,$ a contradiction. Thus, $c\in ann(I).$ Secondly, assume that
$m\notin(0:_{M}I)$. Then there exists $a\in I$ such that $am\neq0$ and this
yields $(0,0)\neq(a,m)(c,m)=(ac,am)\in I\ltimes N$. Thus, we conclude either
$(s,k)(a,m)\in\sqrt{I\ltimes N}$ or $(s,k)(c,m)\in I\ltimes N$ which implies
either $sc\in\sqrt{I}\subseteq\sqrt{(N:_{R}M)}$ or $sm\in N$, so we get a
required contradiction.
\end{proof}

\section{Weakly) S-primary Submodules of Amalgamation Modules}

Let $R$ be a ring, $J$ an ideal of $R$ and $M$ an $R$-module. As a subring of
$R\times R$, In \cite{Danna}, the amalgamated duplication of $R$ along $J$ is
defined as%

\[
R\Join J=\left\{  (r,r+j):r\in R\text{ , }j\in J\right\}
\]
Recently, in \cite{Bouba}, the duplication of the $R$-module $M$ along the
ideal $J$ denoted by $M\Join J$ is defined as%

\[
M\Join J=\left\{  (m,m^{\prime})\in M\times M:m-m^{\prime}\in JM\right\}
\]
which is an $(R\Join J)$-module with scaler multiplication defined by
$(r,r+j).(m,m^{\prime})=(rm,(r+j)m^{\prime})$ for $r\in R$, $j\in J$ and
$(m,m^{\prime})\in M\Join J$. For various properties and results concerning
this kind of modules, one may see \cite{Bouba}.

Let $J$ be an ideal of a ring $R$ and $N$ be a submodule of an $R$-module $M$. Then%

\[
N\Join J=\left\{  (n,m)\in N\times M:n-m\in JM\right\}
\]
and
\[
\bar{N}=\left\{  (m,n)\in M\times N:m-n\in JM\right\}
\]

are clearly submodules of $M\Join J$. If $S$ is a multiplicatively closed
subset of $R$, then the sets $S\Join J=\left\{  (s,s+j):s\in S\text{, }j\in
J\right\}  $ and $\bar{S}=\left\{  (r,r+j):r+j\in S\right\}  $ are obviously
multiplicatively closed subsets of $R\Join J$.

In general, let $f:R_{1}\rightarrow R_{2}$ be a ring homomorphism, $J$ be an
ideal of $R_{2}$, $M_{1}$ be an $R_{1}$-module, $M_{2}$ be an $R_{2}$-module
(which is an $R_{1}$-module induced naturally by $f$) and $\varphi
:M_{1}\rightarrow M_{2}$ be an $R_{1}$-module homomorphism. The subring
\[
R_{1}\Join^{f}J=\left\{  (r,f(r)+j):r\in R_{1}\text{, }j\in J\right\}
\]
of $R_{1}\times R_{2}$ is called the amalgamation of $R_{1}$ and $R_{2}$ along
$J$ with respect to $f$. In \cite{Rachida}, the amalgamation of $M_{1}$ and
$M_{2}$ along $J$ with respect to $\varphi$ is defined as%

\[
M_{1}\Join^{\varphi}JM_{2}=\left\{  (m_{1},\varphi(m_{1})+m_{2}):m_{1}\in
M_{1}\text{ and }m_{2}\in JM_{2}\right\}
\]
which is an $(R_{1}\Join^{f}J)$-module with the scaler product defined as
\[
(r,f(r)+j)(m_{1},\varphi(m_{1})+m_{2})=(rm_{1},\varphi(rm_{1})+f(r)m_{2}%
+j\varphi(m_{1})+jm_{2})
\]
For submodules $N_{1}$ and $N_{2}$ of $M_{1}$ and $M_{2}$, respectively, one
can easily justify that the sets
\[
N_{1}\Join^{\varphi}JM_{2}=\left\{  (m_{1},\varphi(m_{1})+m_{2})\in M_{1}%
\Join^{\varphi}JM_{2}:m_{1}\in N_{1}\right\}
\]
and
\[
\overline{N_{2}}^{\varphi}=\left\{  (m_{1},\varphi(m_{1})+m_{2})\in M_{1}%
\Join^{\varphi}JM_{2}:\text{ }\varphi(m_{1})+m_{2}\in N_{2}\right\}
\]
are submodules of $M_{1}\Join^{\varphi}JM_{2}$. Moreover if $S_{1}$ and
$S_{2}$ are multiplicatively closed subsets of $R_{1}$ and $R_{2}$,
respectively, then
\[
S_{1}\Join^{f}J=\left\{  (s_{1},f(s_{1})+j):s\in S_{1}\text{, \ }j\in
J\right\}
\]
and
\[
\overline{S_{2}}^{\varphi}=\left\{  (r,f(r)+j):r\in R_{1}\text{, \ }f(r)+j\in
S_{2}\right\}
\]
are clearly multiplicatively closed subsets of $R_{1}\Join^{f}J$.

Note that if $R=R_{1}=R_{2}$, $M=M_{1}=M_{2}$, $f=Id_{R}$ and $\varphi=Id_{M}%
$, then the amalgamation of $M_{1}$ and $M_{2}$ along $J$ with respect to
$\varphi$ is exactly the duplication of the $R$-module $M$ along the ideal
$J$. Moreover, in this case, we have $N_{1}\Join^{\varphi}JM_{2}=N\Join J$,
$\overline{N_{2}}^{\varphi}=\bar{N}$, $S_{1}\Join^{f}J=S\Join J$ and
$\overline{S_{2}}^{\varphi}=\bar{S}$.

The proof of the following lemma is straightforward.

\begin{lemma}
\label{HA}Let $M_{1}\Join^{\varphi}JM_{2}$, $N_{1}\Join^{\varphi}JM_{2}$ and
$\overline{N_{2}}^{\varphi}$ be as above. Then

\begin{enumerate}
\item $(r_{1},f(r_{1})+j)\in(N_{1}\Join^{\varphi}JM_{2}:_{R_{1}\Join^{f}%
J}M_{1}\Join^{\varphi}JM_{2})$ if and only if $r_{1}\in(N_{1}:_{R_{1}}M_{1})$.

\item If $f$ and $\varphi$ are epimorphisms, then $(r_{1},f(r_{1}%
)+j)\in(\overline{N_{2}}^{\varphi}:_{R_{1}\Join^{f}J}M_{1}\Join^{\varphi
}JM_{2})$ if and only if $f(r_{1})+j\in(N_{2}:_{R_{2}}M_{2})$.
\end{enumerate}
\end{lemma}

\begin{theorem}
\label{Amalg}Consider the $(R_{1}\Join^{f}J)$-module $M_{1}\Join^{\varphi
}JM_{2}$ defined as above. Let $S$ be a multiplicatively closed subsets of
$R_{1}$ and $N_{1}$ be submodule of $M_{1}$. Then
\end{theorem}

\begin{enumerate}
\item $N_{1}\Join^{\varphi}JM_{2}$ is an $S\Join^{f}J$-primary submodule of
$M_{1}\Join^{\varphi}JM_{2}$ if and only if $N_{1}$ is an $S$-primary
submodule of $M_{1}$.

\item $N_{1}\Join^{\varphi}JM_{2}$ is a weakly $S\Join^{f}J$-primary submodule
of $M_{1}\Join^{\varphi}JM_{2}$ if and only if $N_{1}$ is a weakly $S$-primary
submodule of $M_{1}$ and for $r_{1}\in R_{1}$, $m_{1}\in M_{1}$ with
$r_{1}m_{1}=0$ but $s_{1}r_{1}\notin\sqrt{(N_{1}:_{R_{1}}M_{1})}$ and
$s_{1}m_{1}\notin N_{1}$ for all $s_{1}\in S$, then $f(r_{1})m_{2}+j\phi
(m_{1})+jm_{2}=0$ for every $j\in J$ and $m_{2}\in JM_{2}$.
\end{enumerate}

\begin{proof}
It is easy to verify that $(N_{1}\Join^{\varphi}JM_{2}:_{R_{1}\Join^{f}J}%
M_{1}\Join^{\varphi}JM_{2})\cap(S\Join^{f}J)=\phi$ if and only if
$(N_{1}:_{R_{1}}M_{1})\cap S=\phi$.

(1) $\Longrightarrow)$ Suppose $(s,f(s)+j)$ is a weakly $S\Join^{f}J$-element
of $N_{1}\Join^{\varphi}JM_{2}$. Let $r_{1}\in R_{1}$ and $m_{1}\in M_{1}$
such that $r_{1}m_{1}\in N_{1}$. Then $(r_{1},f(r_{1}))\in R_{1}\Join^{f}J$ ,
$(m_{1},\varphi(m_{1}))\in M_{1}\Join^{\varphi}JM_{2}$ and $(r_{1}%
,f(r_{1}))(m_{1},\varphi(m_{1}))=(r_{1}m_{1},\varphi(r_{1}m_{1}))\in
N_{1}\Join^{\varphi}JM_{2}$. By assumption, we have either
\[
(s,f(s)+j)(r_{1},f(r_{1}))\in\sqrt{(N_{1}\Join^{\varphi}JM_{2}:_{R_{1}%
\Join^{f}J}M_{1}\Join^{\varphi}JM_{2})}%
\]
or%
\[
(s,f(s)+j)(m_{1},\varphi(m_{1}))\in N_{1}\Join^{\varphi}JM_{2}%
\]
In the first case, we conclude by Lemma \ref{HA} that $sr_{1}\in\sqrt
{(N_{1}:_{R_{1}}M_{1})}$. In the second case, we get $sm_{1}\in N_{1}$ and so
$N_{1}$ is an $S$-primary submodule of $M_{1}$.

$\Longleftarrow)$ Let $s$ be a weakly $S$-element of $N_{1}$. Let
$(r_{1},f(r_{1})+j)\in R_{1}\Join^{f}J$ and $(m_{1},\varphi(m_{1})+m_{2})\in
M_{1}\Join^{\varphi}JM_{2}$ such that $(r_{1},f(r_{1})+j)(m_{1},\varphi
(m_{1})+m_{2})\in N_{1}\Join^{\varphi}JM_{2}$. Then $r_{1}m_{1}\in N_{1}$ and
so either $sr_{1}\in\sqrt{(N_{1}:_{R_{1}}M_{1})}$ or $sm_{1}\in N_{1}$. If
$sr_{1}\in\sqrt{(N_{1}:_{R_{1}}M_{1})}$, then by Lemma \ref{HA},
$(s,f(s))(r_{1},f(r_{1})+j)\in\sqrt{(N_{1}\Join^{\varphi}JM_{2}:_{R_{1}%
\Join^{f}J}M_{1}\Join^{\varphi}JM_{2})}$ and if $sm_{1}\in N_{1}$, then
$(s,f(s))(m_{1},\varphi(m_{1})+m_{2})\in N_{1}\Join^{\varphi}JM_{2}$. Thus,
$N_{1}\Join^{\varphi}JM_{2}$ is an $S\Join^{f}J$-primary submodule of
$M_{1}\Join^{\varphi}JM_{2}$ associated to $(s,f(s))\in S\Join^{f}J$.

(2) $\Longrightarrow)$ Suppose $(s,f(s)+j)$ is a weakly $S\Join^{f}J$-element
of $N_{1}\Join^{\varphi}JM_{2}$. Let $r_{1}\in R_{1}$ and $m_{1}\in M_{1}$
such that $0\neq r_{1}m_{1}\in N_{1}$. Then $(0,0)\neq(r_{1},f(r_{1}%
))(m_{1},\varphi(m_{1}))=(r_{1}m_{1},\varphi(r_{1}m_{1}))\in N_{1}%
\Join^{\varphi}JM_{2}$. By assumption, either $(s,f(s)+j)(r_{1},f(r_{1}%
))\in\sqrt{(N_{1}\Join^{\varphi}JM_{2}:_{R_{1}\Join^{f}J}M_{1}\Join^{\varphi
}JM_{2})}$ or $(s,f(s)+j)(m_{1},\varphi(m_{1}))\in N_{1}\Join^{\varphi}JM_{2}%
$. Thus, $sr_{1}\in\sqrt{(N_{1}:_{R_{1}}M_{1})}$ by Lemma \ref{HA} or
$sm_{1}\in N_{1}$ and so $N_{1}$ is weakly $S$-primary in $M_{1}$. We use the
contrapositive to prove the other part. Let $r_{1}\in R_{1}$, $m_{1}\in M_{1}$
with $r_{1}m_{1}=0$ and $f(r_{1})m_{2}+j\phi(m_{1})+jm_{2}\neq0$ for some
$j\in J$ and some $m_{2}\in JM_{2}$. Then
\begin{align*}
(0,0)  &  \neq(r_{1},f(r_{1})+j)(m_{1},\varphi(m_{1})+m_{2})\\
&  =(0,f(r_{1})m_{2}+j\varphi(m_{1})+jm_{2})\in N_{1}\Join^{\varphi}JM_{2}%
\end{align*}
By assumption, either $(s,f(s)+j)(r_{1},f(r_{1})+j)\in\sqrt{(N_{1}%
\Join^{\varphi}JM_{2}:_{R_{1}\Join^{f}J}M_{1}\Join^{\varphi}JM_{2})}$ or
$(s,f(s)+j)(m_{1},\varphi(m_{1})+m_{2})\in N_{1}\Join^{\varphi}JM_{2}$ and so
again $sr_{1}\in\sqrt{(N_{1}:_{R_{1}}M_{1})}$ or $sm_{1}\in N_{1}$ as needed.

$\Longleftarrow)$ Let $s$ be a weakly $S$-element of $N_{1}$, $(r_{1}%
,f(r_{1})+j)\in R_{1}\Join^{f}J$ and $(m_{1},\varphi(m_{1})+m_{2})\in
M_{1}\Join^{\varphi}JM_{2}$ such that
\begin{align*}
(0,0)  &  \neq(r_{1}m_{1},\varphi(r_{1}m_{1})+f(r_{1})m_{2}+j\varphi
(m_{1})+jm_{2})\\
&  =(r_{1},f(r_{1})+j)(m_{1},\varphi(m_{1})+m_{2})\in N_{1}\Join^{\varphi
}JM_{2}%
\end{align*}
If $0\neq r_{1}m_{1}$, then the proof is similar to that of (1). Suppose
$r_{1}m_{1}=0$. Then $f(r_{1})m_{2}+j\varphi(m_{1})+jm_{2}\neq0$ and so by
assumption there exists $s^{\prime}\in S$ such that either $s^{\prime}r_{1}%
\in\sqrt{(N_{1}:_{R_{1}}M_{1})}$ or $s^{\prime}m_{1}\in N_{1}$. Thus,
$(s^{\prime},f(s^{\prime}))(r_{1},f(r_{1})+j)\in\sqrt{(N_{1}\Join^{\varphi
}JM_{2}:_{R_{1}\Join^{f}J}M_{1}\Join^{\varphi}JM_{2})}$ or $(s^{\prime
},f(s^{\prime}))(m_{1},\varphi(m_{1})+m_{2})\in N_{1}\Join^{\varphi}JM_{2}$.
Therefore, $N_{1}\Join^{\varphi}JM_{2}$ is a weakly $S\Join^{f}J$-primary
submodule of $M_{1}\Join^{\varphi}JM_{2}$ associated to $(ss^{\prime
},f(ss^{\prime}))\in S\Join^{f}J$.
\end{proof}

In particular, if we take $S=\left\{  1_{R_{1}}\right\}  $ and consider
$S\Join^{f}0=\left\{  (1_{R_{1}},1_{R_{2}})\right\}  $) in Theorem
\ref{Amalg}, then we get the following corollary.

\begin{corollary}
\label{ca1}Consider the $(R_{1}\Join^{f}J)$-module $M_{1}\Join^{\varphi}%
JM_{2}$ defined as in Theorem \ref{Amalg} and let $N_{1}$ be a submodule of
$M_{1}$. Then
\end{corollary}

\begin{enumerate}
\item $N_{1}\Join^{\varphi}JM_{2}$ is a primary submodule of $M_{1}%
\Join^{\varphi}JM_{2}$ if and only if $N_{1}$ is a primary submodule of
$M_{1}$.

\item $N_{1}\Join^{\varphi}JM_{2}$ is a weakly primary submodule of
$M_{1}\Join^{\varphi}JM_{2}$ if and only if $N_{1}$ is a weakly primary
submodule of $M_{1}$ and for $r_{1}\in R_{1}$, $m_{1}\in M_{1}$ with
$r_{1}m_{1}=0$ but $r_{1}\notin\sqrt{(N_{1}:_{R_{1}}M_{1})}$ and $m_{1}\notin
N_{1}$, then $f(r_{1})m_{2}+j\phi(m_{1})+jm_{2}=0$ for every $j\in J$ and
$m_{2}\in JM_{2}$.
\end{enumerate}

\begin{theorem}
\label{Amalg2}Consider the $(R_{1}\Join^{f}J)$-module $M_{1}\Join^{\varphi
}JM_{2}$ defined as in Theorem \ref{Amalg} where $f$ and $\varphi$ are
epimorphisms. Let $S$ be a multiplicatively closed subsets of $R_{2}$ and
$N_{2}$ be a submodule of $M_{2}$. Then
\end{theorem}

\begin{enumerate}
\item $\overline{N_{2}}^{\varphi}$ is an $\overline{S}^{\varphi}$-primary
submodule of $M_{1}\Join^{\varphi}JM_{2}$ if and only if $N_{2}$ is an
$S$-primary submodule of $M_{2}$.

\item $\overline{N_{2}}^{\varphi}$ is a weakly $\overline{S}^{\varphi}%
$-primary submodule of $M_{1}\Join^{\varphi}JM_{2}$ if and only if $N_{2}$ is
a weakly $S$-primary submodule of $M_{2}$ and for $r_{1}\in R_{1}$, $m_{1}\in
M_{1}$, $m_{2}\in JM_{2}$, $j\in J$ with $(f(r_{1})+j)(\varphi(m_{1}%
)+m_{2})=0$ but $s(f(r_{1})+j)\notin\sqrt{(N_{2}:_{R_{2}}M_{2})}$ and
$s(\varphi(m_{1})+m_{2})\notin N_{2}$ for all $s\in S$, then $r_{1}m_{1}=0$.
\end{enumerate}

\begin{proof}
In view of Lemma \ref{HA}, we can easily prove that $(\overline{N_{2}%
}^{\varphi}:_{R_{1}\Join^{f}J}M_{1}\Join^{\varphi}JM_{2})\cap\overline
{S}^{\varphi}=\phi$ if and only if $(N_{2}:_{R_{2}}M_{2})\cap S=\phi$.

(1). $\Longrightarrow)$ Suppose $N_{2}$ is an $S$-primary submodule of $M_{2}$
associated to $s=f(t)\in S$. Let $(r_{1},f(r_{1})+j)\in R_{1}\Join^{f}J$ and
$(m_{1},\varphi(m_{1})+m_{2})\in M_{1}\Join JM_{2}$ such that $(r_{1}%
,f(r_{1})+j)(m_{1},\varphi(m_{1})+m_{2})\in\overline{N_{2}}^{\varphi}$. Then
$(f(r_{1})+j)(\varphi(m_{1})+m_{2})\in N_{2}$ and so $s(f(r_{1})+j)\in
\sqrt{(N_{2}:_{R_{2}}M_{2})}$ or $s(\varphi(m_{1})+m_{2})\in N_{2}$. If
$s(f(r_{1})+j)\in\sqrt{(N_{2}:_{R_{2}}M_{2})}$, then by Lemma \ref{HA}, we
have $(t,s)(r_{1},f(r_{1})+j)\in\sqrt{(\overline{N_{2}}^{\varphi}:_{R_{1}%
\Join^{f}J}M_{1}\Join^{\varphi}JM_{2})}$. If $s(\varphi(m_{1})+m_{2})\in
N_{2}$, then $(t,s)(m_{1},\varphi(m_{1})+m_{2})\in\overline{N_{2}}^{\varphi}$.
Therefore, $\overline{N_{2}}^{\varphi}$ is $\overline{S}^{\varphi}$-primary in
$M_{1}\Join^{\varphi}JM_{2}$.

$\Longleftarrow)$ Suppose $\overline{N_{2}}^{\varphi}$ is an $\overline
{S}^{\varphi}$-primary submodule of $M_{1}\Join^{\varphi}JM_{2}$ associated to
$(t,f(t)+j)=(t,s)\in\overline{S}^{\varphi}$. Let $r_{2}=f(r_{1})\in R_{2}$ and
$m_{2}=\varphi(m_{1})\in M_{2}$ such that $r_{2}m_{2}\in N_{2}$. Then
$(r_{1},r_{2})(m_{1},m_{2})\in\overline{N_{2}}^{\varphi}$ where $(r_{1}%
,r_{2})\in R_{1}\Join^{f}J$ and $(m_{1},m_{2})\in M_{1}\Join^{\varphi}JM_{2}$.
By assumption, $(t,s)(r_{1},r_{2})\in\sqrt{(\overline{N_{2}}^{\varphi}%
:_{R_{1}\Join^{f}J}M_{1}\Join^{\varphi}JM_{2})}$ or $(t,s)(m_{1},m_{2}%
)\in\overline{N_{2}}^{\varphi}$. In the first case, Lemma \ref{HA} implies
$sr_{2}\in\sqrt{(N_{2}:_{R_{2}}M_{2})}$. In the second case, we conclude
$sm_{2}\in N_{2}$ and the result follows.

(2). $\Longrightarrow)$ Let $(t,f(t)+j)=(t,s)$ be a weakly $\overline
{S}^{\varphi}$-element of $\overline{N_{2}}^{\varphi}$. Let $r_{2}=f(r_{1})\in
R_{2}$ and $m_{2}=f(m_{1})\in M_{2}$ such that $0\neq r_{2}m_{2}\in N_{2}$.
Then $(0.0)\neq(r_{1},r_{2})(m_{1},m_{2})\in\overline{N_{2}}^{\varphi}$ where
$(r_{1},r_{2})\in R_{1}\Join^{f}J$ and $(m_{1},m_{2})\in M_{1}\Join^{\varphi
}JM_{2}$. Thus, either $(t,s)(r_{1},r_{2})\in\sqrt{(\overline{N_{2}}^{\varphi
}:_{R_{1}\Join^{f}J}M_{1}\Join^{\varphi}JM_{2})}$ or $(t,s)(m_{1},m_{2}%
)\in\overline{N_{2}}^{\varphi}$. Hence, either $sr_{2}\in\sqrt{(N_{2}:_{R_{2}%
}M_{2})}$ by Lemma \ref{HA} or $sm_{2}\in N_{2}$ and $s$ is a weakly
$S$-element of $N_{2}$. For the other part, we use contrapositive. Let
$r_{1}\in R_{1}$, $m_{1}\in M_{1}$, $m_{2}\in JM_{2}$, $j\in J$ with
$(f(r_{1})+j)(\varphi(m_{1})+m_{2})=0$ and suppose $r_{1}m_{1}\neq0$. Then
$(0,0)\neq(r_{1},f(r_{1})+j)(m_{1},\varphi(m_{1})+m_{2})\in\overline{N_{2}%
}^{\varphi}$ and so $(t,s)(r_{1},f(r_{1})+j)\in\sqrt{(\overline{N_{2}%
}^{\varphi}:_{R_{1}\Join^{f}J}M_{1}\Join JM_{2})}$ or $(t,s)(m_{1}%
,\varphi(m_{1})+m_{2})\in\overline{N_{2}}^{\varphi}$. Hence, either
$s(f(r_{1})+j)\in\sqrt{(N_{2}:_{R_{2}}M_{2})}$ or $s(\varphi(m_{1})+m_{2})\in
N_{1}$ and the result follows.

$\Longleftarrow)$ Suppose $s=f(t)\in S$ is a weakly $S$-element of $N_{2}$.
Let $(r_{1},f(r_{1})+j)\in R_{1}\Join^{f}J$ and $(m_{1},\varphi(m_{1}%
)+m_{2})\in M_{1}\Join^{\varphi}JM_{2}$ such that $(0,0)\neq(r_{1}%
,f(r_{1})+j)(m_{1},\varphi(m_{1})+m_{2})\in\overline{N_{2}}^{\varphi}$. Then
$(f(r_{1})+j)(\varphi(m_{1})+m_{2})\in N_{2}$. If $(f(r_{1})+j)(\varphi
(m_{1})+m_{2})=0$, then $r_{1}m_{1}\neq0$. So by assumption, there exists
$s^{\prime}=f(t^{\prime})\in S$ such that $s^{\prime}(f(r_{1})+j)\in
\sqrt{(N_{2}:_{R_{2}}M_{2})}$ or $s^{\prime}(\varphi(m_{1})+m_{2})\in N_{2}$.
It follows that $(t^{\prime},s^{\prime})(r_{1},f(r_{1})+j)\in\sqrt
{(\overline{N_{2}}^{\varphi}:_{R_{1}\Join^{f}J}M_{1}\Join^{\varphi}JM_{2})}$
or $(t^{\prime},s^{\prime})(m_{1},\varphi(m_{1})+m_{2})\in\overline{N_{2}%
}^{\varphi}$. Hence, $\overline{N_{2}}^{\varphi}$ is a weakly $\overline
{S}^{\varphi}$-primary submodule of $M_{1}\Join^{\varphi}JM_{2}$ associated to
$(tt^{\prime},ss^{\prime})$. If $(f(r_{1})+j)(\varphi(m_{1})+m_{2})\neq0$,
then the result follows as in the proof of (1).
\end{proof}

In particular, if we take $S=\left\{  1_{R_{2}}\right\}  $ and consider the
multiplicatively closed subset $\overline{S}^{\varphi}=\left\{  (1_{R_{1}%
},1_{R_{2}})\right\}  $ of $M_{1}\Join^{\varphi}JM_{2}$ in Theorem
\ref{Amalg2}, then we get the following corollary.

\begin{corollary}
\label{ca2}Let $M_{1}\Join^{\varphi}JM_{2}$ and $N_{2}$ be defined as in
Theorem \ref{Amalg2}. Then
\end{corollary}

\begin{enumerate}
\item $\overline{N_{2}}^{\varphi}$ is a primary submodule of $M_{1}%
\Join^{\varphi}JM_{2}$ if and only if $N_{2}$ is a primary submodule of
$M_{2}$.

\item $\overline{N_{2}}^{\varphi}$ is a weakly primary submodule of
$M_{1}\Join^{\varphi}JM_{2}$ if and only if $N_{2}$ is a weakly primary
submodule of $M_{2}$ and for $r_{1}\in R_{1}$, $m_{1}\in M_{1}$, $m_{2}\in
JM_{2}$, $j\in J$ with $(f(r_{1})+j)(\varphi(m_{1})+m_{2})=0$ but
$(f(r_{1})+j)\notin\sqrt{(N_{2}:_{R_{2}}M_{2})}$ and $(\varphi(m_{1}%
)+m_{2})\notin N_{2}$, then $r_{1}m_{1}=0$.
\end{enumerate}

\begin{corollary}
\label{Dup}Let $N$ be a submodule of an $R$-module $M$, $J$ an ideal of $R$
and $S$ a multiplicatively closed subset of $R$. The following are equivalent:
\end{corollary}

\begin{enumerate}
\item $N$ is an $S$-primary submodule of $M$.

\item $N\Join J$ is an $(S\Join J)$-primary submodule of $M\Join J$.

\item $\overline{N}$ is an $\overline{S}$-primary submodule of $M\Join J$.
\end{enumerate}

\begin{corollary}
\label{Dup1}Let $N$ be a submodule of an $R$-module $M$, $J$ an ideal of $R$
and $S$ a multiplicatively closed subset of $R$. The following are equivalent:
\end{corollary}

\begin{enumerate}
\item $N\Join J$ is a weakly $(S\Join J)$-primary submodule of $M\Join J$

\item $N$ is a weakly $S$-primary submodule of $M$ and for $r\in R$, $m\in M$
with $rm=0$ but $sr\notin\sqrt{(N:_{R_{1}}M)}$ and $sm\notin N$ for all $s\in
S$, then $(r+j)m^{\prime}=0$ for every $j\in J$ and $m^{\prime}\in JM_{2}$.
\end{enumerate}

\begin{corollary}
\label{Dup2}Let $N$ be a submodule of an $R$-module $M$, $J$ an ideal of $R$
and $S$ a multiplicatively closed subset of $R$.
\end{corollary}

\begin{enumerate}
\item $\overline{N}$ is a weakly $\overline{S}$-prime submodule of $M\Join J$.

\item $N$ is a weakly $S$-prime submodule of $M$ and for $r\in R$, $m\in M$,
$m^{\prime}\in JM$, $j\in J$ with $(r+j)(m+m^{\prime})=0$ but $s(r+j)\notin
(N:_{R}M)$ and $s(m+m^{\prime})\notin N$ for all $s\in S$, then $rm=0$.
\end{enumerate}

Next, we justify that the second condition of Corollary \ref{Dup1} (2) can not
be ignored.

\begin{example}
\label{ex2}Consider the ideal $J=2%
\mathbb{Z}
$ of $%
\mathbb{Z}
$ and the submodule $N=0\times\left\langle \bar{0}\right\rangle $ of the $%
\mathbb{Z}
$-module $M=%
\mathbb{Z}
\times%
\mathbb{Z}
_{6}$. Then $M\Join J=\left\{  (m,m^{\prime})\in M\times M:m-m^{\prime}\in
JM=2%
\mathbb{Z}
\times\left\langle \bar{2}\right\rangle \right\}  $ and $N\Join J=\left\{
(n,m)\in N\times M:n-m\in2%
\mathbb{Z}
\times\left\langle \bar{2}\right\rangle \right\}  $. Obviously, $N$ is a
weakly primary submodule of $M$. On the other hand, $N\Join J$ is not a weakly
primary submodule of $M\Join J$. Indeed, $(2,4)\in%
\mathbb{Z}
\Join J$ and $((0,\bar{3}),(0,\bar{1}))\in M\Join J$ with $(2,4).((0,\bar
{3}),(0,\bar{1}))=((0,\bar{0}),(0,\bar{4}))\in N\Join J$. But $(2,4)\notin
\sqrt{((N\Join J):_{%
\mathbb{Z}
\Join I}(M\Join J))}$ as $2\notin\sqrt{(N:_{%
\mathbb{Z}
}M)}=\left\langle 0\right\rangle $ and $((0,\bar{3}),(0,\bar{1}))\notin N\Join
J$. We note that if we take $r=2$ and $m=(0,\bar{3})\in M$, then clearly,
$rm=0$, $r\notin\sqrt{(N:_{R}M)}=0$ and $m\notin N$ but for $m^{\prime
}=(0,\bar{2})\in JM=2%
\mathbb{Z}
\times\left\langle \bar{2}\right\rangle $, we have $(r+0)m^{\prime}\neq0$.
\end{example}

Also, if the second condition of (2) in Corollary \ref{Dup2} does not hold,
then we may find a weakly $S$-primary submodule $N$ of $M$ such that
$\overline{N}$ is not a weakly $\overline{S}$-primary submodule of $M\Join J$.

\begin{example}
Let $N$, $M$ and $J$ be as in Example \ref{ex2}. Choose $(2,4)\in%
\mathbb{Z}
\Join J$ and $((0,\bar{1}),(0,\bar{3}))\in M\Join J$. Then we have
$(2,4).((0,\bar{1}),(0,\bar{3}))\in\overline{N}$ but clearly $(2,4)\notin
\sqrt{(\bar{N}:_{%
\mathbb{Z}
\Join I}(M\Join J))}$ and $((0,\bar{1}),(0,\bar{3}))\notin\overline{N}$.
Therefore, $\bar{N}$ is not a weakly primary submodule of $M\Join J$.
\end{example}

\end{document}